\documentclass[12 pt]{article}
\usepackage{amsmath}
\usepackage{amsfonts}
\usepackage{amssymb}
\usepackage{amsthm}
\usepackage{thmtools}
\usepackage{hyperref}
\usepackage{mathtools}
\usepackage{enumitem}
\usepackage[style=numeric,backend=bibtex]{biblatex}

\newtheorem{thm}{Theorem}[section]
\newtheorem{lem}[thm]{Lemma}

\newtheorem{cor}[thm]{Corollary}
\newcommand\numberthis{\addtocounter{equation}{1}\tag{\theequation}}

\allowdisplaybreaks

\title{Alternating odd pretzel knots and chirally cosmetic surgeries}
\author{Konstantinos Varvarezos}
\begin{document}
\maketitle

\begin{abstract}
A pair of surgeries on a knot is chirally cosmetic if they result in homeomorphic manifolds with opposite orientations.  Using recent methods of Ichihara, Ito, and Saito, we show that, except for the (2,5) and (2,7)-torus knots, the genus 2 and 3 alternating odd pretzel knots do not admit any chirally cosmetic surgeries.  Further, we show that for a fixed genus, at most finitely many alternating odd pretzel knots admit chirally cosmetic surgeries.
\end{abstract}


\section{Introduction}

Given a knot $K$ in $S^3$ and an $r\in \mathbf{Q}\cup \{\infty\},$ we denote the \textit{Dehn surgery} on $K$ with slope $r$ by $S^3_r(K).$  Surgeries on $K$ along distinct slopes $r$ and $r'$ are called \textit{cosmetic} if $S^3_r(K)$ and $S^3_{r'}(K)$ are homeomorphic manifolds.  Furthermore, a pair of such surgeries is said to be \textit{purely cosmetic} if $S^3_r(K)$ and $S^3_{r'}(K)$ are homeomorphic as \textit{oriented} manifolds.  We use the symbol $\cong$ to denote ``orientation-preserving homeomorphic." If, on the other hand, $S^3_r(K) \cong -S^3_{r'}(K),$ we say this pair of surgeries is \textit{chirally cosmetic}; here $-M$ denotes the manifold $M$ with reversed orientation.

No purely cosmetic surgeries are have been found on nontrivial knots in $S^3$; indeed, Problem 1.81(A) in \cite{KirbyList} conjectures that none exist.  On the other hand, there are examples of chirally cosmetic surgeries.  For instance, $S^3_r(K) \cong -S^3_{-r}(K)$ whenever $K$ is an amphicheiral knot.  Also, $(2,n)$-torus knots are known to admit chirally cosmetic surgeries; see \cite{IIS,Mat}.

In this work, we consider the family of \textit{alternating odd pretzel knots}.  These are pretzel knots of the form $P(2k_1+1,2k_2+1,\dots,2k_{2g+1}+1)$, where the $k_i$ are integers all with the same sign; hence, up to mirror image, we may assume all $k_i$ are nonnegative.  See Figure \ref{fig:Pknot} for the five-stranded case.  It is a fact that $g$ is the Seifert genus of the knot.  For convenience, we will often use the shorthand $K(k_1,k_2,\dots,k_{2g+1})$ to denote these knots.

\begin{figure}
\centering
\includegraphics[width=.75\textwidth]{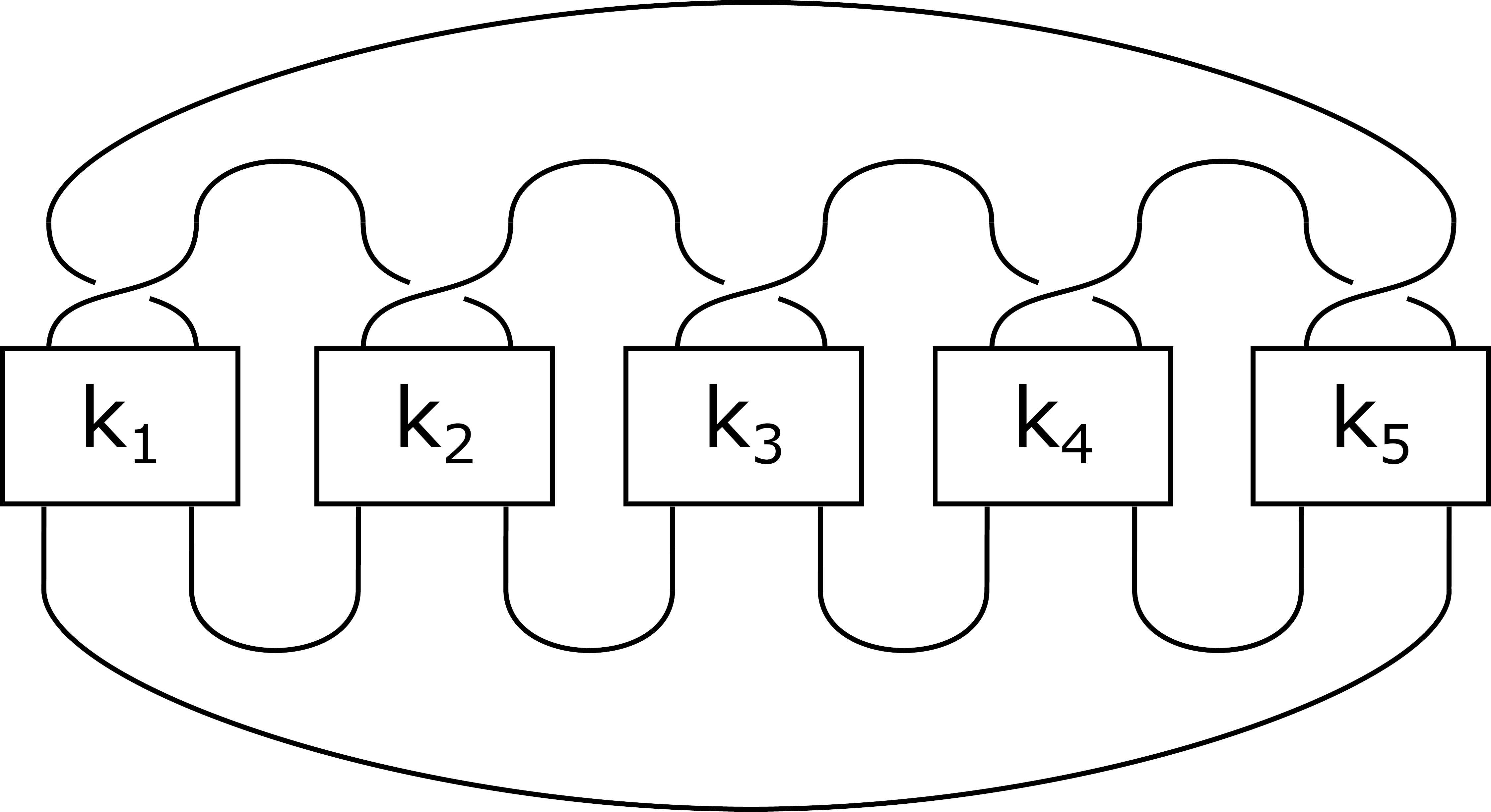}
\caption{The pretzel knot $P(2k_1+1,2k_2+1,2k_3+1,2k_4+1,2k_5+1).$  Here the boxes represent $k_i$ full right-handed twists of the strands passing through.}
\label{fig:Pknot}
\end{figure}

Note that these knots are already known not to admit any purely cosmetic surgeries either.  This can be seen, for instance, from the fact that of the signatures and $v_3$ of these knots are nonzero(see section \ref{sec:comp} below).  By \cite{Hans}, alternating knots that admit purely cosmetic surgeries must have zero signature.  Alternatively, by \cite{IW}, knots admitting purely cosmetic surgeries must have $v_3=0.$

Our first main result is that almost none of the alternating odd pretzel knots of genus 2 and 3 admit chirally cosmetic surgeries.

\begin{restatable}{thm}{main}\label{thm:main}
\begin{enumerate}[label=(\roman*)]
\item Let $K=K(k_1,k_2,k_3,k_4,k_5)$ with each $k_i \geq 0.$  If at least one $k_i>0,$ then $K$ does not admit any chirally cosmetic surgeries.
\item Let $K=K(k_1,k_2,k_3,k_4,k_5,k_6,k_7)$ with each $k_i \geq 0.$  If at least one $k_i>0,$ then $K$ does not admit any chirally cosmetic surgeries.
\end{enumerate}
\end{restatable}
\begin{proof}[Remark.]\let\qed\relax
$K(0,0,0,0,0)$ is the (-2,5)-torus knot, and $K(0,0,0,0,0,0,0)$ is the (-2,7)-torus knot. Both of these are already known to admit (infinitely many pairs of) chirally cosmetic surgeries; see \cite{Mat} and \cite{IIS}.
\end{proof}
Ichihara, Ito, and Saito have already shown the analogue of this for the genus 1 (three-stranded) case.  Indeed, they classified chirally cosmetic surgeries for all alternating genus 1 knots \cite{IIS}.

For the case of higher genus, we show more generally that, given a fixed genus, all but finitely many alternating odd pretzel knots admit no chirally cosmetic surgeries.  More precisely, we have the following:

\begin{restatable}{thm}{gen}\label{thm:gen}
Let $K=K(k_1,\dots,k_{2g+1})$.  If $k_1+\dots+k_{2g+1} \geq \alpha g$ then $K$ does not admit any chirally cosmetic surgeries.  Here $\alpha = \frac{9+\sqrt{237}}{12} \approx 2.0329$.
\end{restatable}
\begin{proof}[Note:]\let\qed\relax
The crossing number of $K(k_1,\dots,k_{2g+1})$ is $2(k_1+\dots+k_{2g+1})+2g+1$ and so one may equivalently rephrase the theorem as saying that such knots admit no chirally cosmetic surgeries whenever the crossing number is at least $2(\alpha+1)g+1$.
\end{proof}

\section{Finite Type Invariants}
Here we briefly recall some facts about finite type invariants (also called Vassiliev invariants) for knots.  Suppose a knot invariant $v$ can be extended to an invariant of \textit{singular knots} (i.e., knots with possibly finitely many points of self-intersection) in a way that satisfies the following:
\[
v(K_{\bullet}) = v(K_+)-v(K_-)
\]
whenever the (singular) knots $K_+$, $K_-$, and $K_{\bullet}$ differ locally near a crossing/self-intersection as in Figure \ref{fig:skein}.  Then $v$ is said to be a finite type invariant of order $n$ if moreover $v(K)=0$ whenever $K$ has at least $n+1$ self-intersection points.

\begin{figure}
\centering
\includegraphics[width=.75\textwidth]{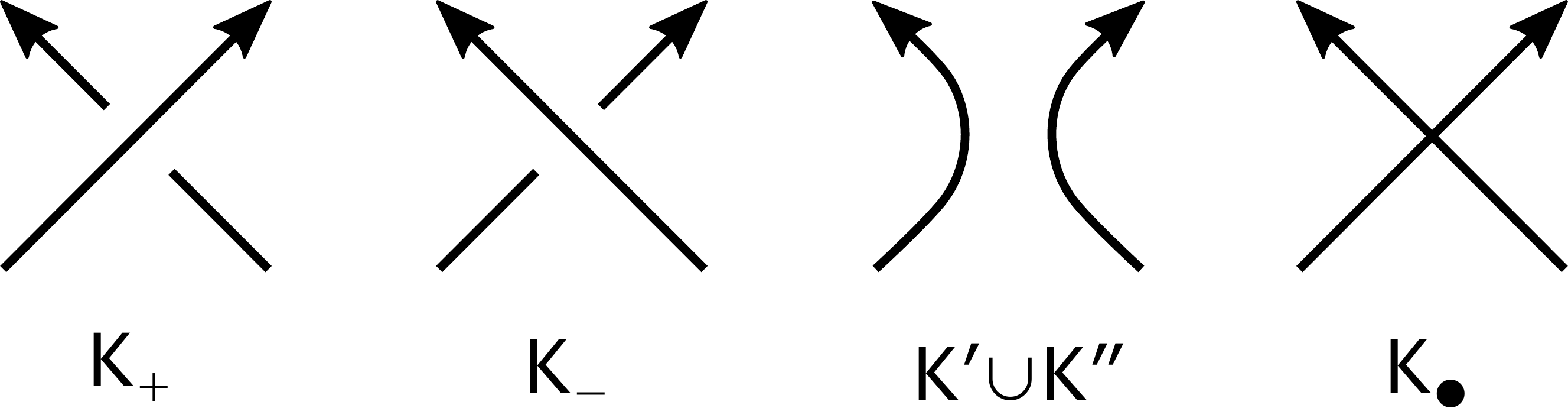}
\caption{The diagrams of $K_+$, $K_-$, $K'\cup K''$, and $K_{\bullet}$ differ only locally near a crossing.}
\label{fig:skein}
\end{figure}

For example, recall the Conway polynomial of a knot, which is related to the Alexander polynomial in the following way:
\[
\nabla_K(z) = \Delta_K(t)|_{z=t^{1/2}-t^{-1/2}}
\]
For a knot $K$, the Conway polynomial will have the form:
\[
\nabla_K(z) = 1 + \sum_{j=1}^{n}a_{2j}(K)z^{2j}
\]
It is a fact that the coefficent $a_{2j}(K)$ is a finite type invariant of order $2j$ for each $j$.

The other finite type invariant that will be of interest to us is $v_3$.  This is a third-order invariant, which may be defined as:
\[
v_3(K) = -\frac{1}{36}V'''_{K}(1) -\frac{1}{12}V''_{K}(1)
\]
where $V_K(t)$ is the Jones polynomial of $K$.
\subsection{Computing the Invariants for Pretzel Knots}\label{sec:comp}
We consider the case $K=K(k_1,k_2,k_3,k_4,k_5)$ first.  Applying Seifert's algorithm to a diagram such as the one in Figure \ref{fig:Pknot}, one finds that the Seifert form for $K$ may be given by the following matrix:
\[
A=\left(\begin{matrix}
k_1+k_2+1 & k_2 & 0 & 0 \\
k_2+1 & k_2+k_3+1 & k_3 & 0 \\
0 & k_3+1 & k_3+k_4+1 & k_4 \\
0 & 0 & k_4+1 & k_4+k_5+1
\end{matrix}\right)
\]
From this, we may compute the Alexander and Conway polynomials.  In particular, $\Delta_K(t)=\det\left(t^{1/2}A-t^{-1/2}A^T\right)$ and $\nabla_K(z) = \Delta_K(t)|_{z=t^{1/2}-t^{-1/2}} = 1+ a_2(K)z^2+a_4(K)z^4$.  Explicit computation gives:
\begin{equation}\label{eq:a242}
\begin{cases}
a_2(K) = 3 + 2 s_{1,5} + s_{2,5}\\
a_4(K) = 1 + s_{1,5} + s_{2,5} + s_{3,5} + s_{4,5}
\end{cases}
\end{equation}
Here we use $s_{n,m}$ as a shorthand for $s_n(k_1,\dots,k_m),$ the $n$th elementary symmetric polynomial in $k_1,\dots,k_m$ given by:
\[
s_{n,m}=\sum_{\substack{P\subseteq\{1,\dots,m\} \\ |P|=n}} \prod_{j\in P} k_j
\]
Let us here record some properties of the elementary symmetric polynomials, which are straightforward consequences of their definition.
\begin{lem}\label{lem:sym}
Let $s_{n,m}$ denote the $n$th elementary symmetric polynomial in $k_1,\dots,k_m$.  Then, the following hold:
\begin{itemize}
\item If $k_j=0$ for all $j>N$, then whenever $n\geq N$, $s_{n,m} = s_{n,N}$
\item $s_{n,m+1} = s_{n,m} + k_{m+1}s_{n-1,m}$
\end{itemize}
\end{lem}
\begin{proof}[Remark.]
By convention, we take $s_{0,m}=1$, and if $n>m$, $s_{n.m}=0$.
\end{proof}

Now we turn to the case $K=K(k_1,k_2,k_3,k_4,k_5,k_6,k_7)$.  As before, we can represtent the Seifert form with the following matrix:
\[
A=\left(\begin{matrix}
k_1+k_2+1 & k_2 & 0 & 0 & 0 & 0 \\
k_2+1 & k_2+k_3+1 & k_3 & 0 & 0 & 0 \\
0 & k_3+1 & k_3+k_4+1 & k_4 & 0 & 0 \\
0 & 0 & k_4+1 & k_4+k_5+1 & k_5 & 0 \\
0 & 0 & 0 & k_5+1 & k_5+k_6+1 & k_6 \\
0 & 0 & 0 & 0 & k_6+1 & k_6+k_7+1
\end{matrix}\right)
\]
Once again, computation of the Conway polynomial gives:
\begin{equation}\label{eq:a243}
\begin{cases}
a_2(K) = 6 + 3 s_{1,7} + s_{2,7}\\
a_4(K) = 5 + 4s_{1,5} + 3s_{2,5} + 2s_{3,5} + s_{4,5}\\
a_6(K) = 1 + s_{1,7} + s_{2,7} + s_{3,7} + s_{4,7} + s_{5,7} + s_{6,7}
\end{cases}
\end{equation}
In the general case where $K=K(k_1,\dots,k_{2g+1})$, we will need the following formulae for $a_2(K)$ and $v_3(K)$.
\begin{lem}\label{lem:av}
Let $K=K(k_1,\dots,k_{2g+1})$.  Then
\begin{align*}
a_2(K)&=\frac{1}{2}g(g+1) + g s_{1,2g+1} + s_{2,2g+1}
\\
v_3(K)&=-\frac{1}{2}\left(\frac{g(g+1)(2g+1)}{3} + g(2g+1)s_{1,2g+1} + g s_{1,2g+1}^2 + 2g s_{2,2g+1} + s_{1,2g+1}s_{2,2g+1} + s_{3,2g+1} \right)
\end{align*}
\end{lem}
\begin{proof}
We employ the following skein relation for $a_2$ \cite{IW}:
\begin{equation}\label{eq:skein2}
a_2(K_+)-a_2(K_-)=\mathrm{lk}(K',K'')
\end{equation}
Here $K_+$ and $K_-$ differ at a single crossing, and $K'$ and $K''$ are the two components of the link that results from the oriented resolution at that crossing; see Figure \ref{fig:skein} for an illustration.

We proceed by induction on $g$.  Notice that $g=0$ corresponds to the unknot, which has $a_2=0$.
Now consider $g \geq 1$
Considering one of the crossings on the $2g+1$st strand, we have $K_{-}=K(k_1,\dots,k_{2g+1})$, $K_{+}=K(k_1,\dots,k_{2g+1}-1),$ and $K' \cup K''$ is a pair of unknots with linking number $-(k_1+\dots+k_{2g}+g) = -(s_{1,2g}+g)$.  Applying (\ref{eq:skein2}) gives:
\[
a_2(K(k_1,\dots,k_{2g+1}))-a_2(K(k_1,\dots,k_{2g+1}-1)) = s_{1,2g}+g
\]
Repeating this procedure $k_{2g+1}-1$  times gives:
\[
a_2(K(k_1,\dots,k_{2g+1}))-a_2(K(k_1,\dots,k_{2g},0)) = k_{2g+1}(s_{1,2g}+g)
\]
Now repeating this procedure for the $2g$th strand gives:
\begin{multline}\label{eq:ind2}
a_2(K(k_1,\dots,k_{2g+1}))-a_2(K(k_1,\dots,k_{2g-1},0,0)) \\= k_{2g}(s_{1,2g-1}+g) + k_{2g+1}(s_{1,2g}+g)
\end{multline}
Now, applying (\ref{eq:skein2}) to one of the two rightmost crossings of $K(k_1\dots,k_{2g-1},0,0)$, we see that:
\[
a_2(K(k_1,\dots,k_{2g-1},0,0)) - a_2(K(k_1,\dots,k_{2g-1})) = (s_{1,2g-1}+g)
\]
Combining this with (\ref{eq:ind2}) and applying the induction hypothesis, we find:
\begin{align*}
a_2(K&(k_1,\dots,k_{2g+1})) = a_2(K(k_1,\dots,k_{2g-1}))  + (k_{2g}+1)(s_{1,2g-1}+g) + k_{2g+1}(s_{1,2g}+g) \\
&= \frac{1}{2}g(g-1) + (g-1) s_{1,2g-1} + s_{2,2g-1} + (k_{2g}+1)(s_{1,2g-1}+g) + k_{2g+1}(s_{1,2g}+g)\\
&= \frac{1}{2}g(g-1) + g + (g-1) s_{1,2g-1} + s_{1,2g-1} + gk_{2g}+ gk_{2g+1} +  s_{2,2g-1} + k_{2g}s_{1,2g-1} + k_{2g+1}s_{1,2g}\\
&=\frac{1}{2}g(g+1) + gs_{1,2g+1} + s_{2,2g+1}
\end{align*}
where Lemma \ref{lem:sym} was used in the last step.

Let us now compute $v_3(K).$  We shall make use of the fact that $v_3$ satisfies the following skein relation \cite{IW,Lescop}:
\begin{equation}\label{eq:skein3}
v_3(K_+) - v_3(K_-) = \frac{1}{2}\Big(a_2(K_+)+a_2(K_-)+\mathrm{lk}^2(K',K'')\Big)-a_2(K')-a_2(K'')
\end{equation}
Again we proceed by induction, noting that for $g=0$, $v_3$ of the unknot is zero.  Suppose $g \geq 1$.  As before, we use (\ref{eq:skein3}) on a crossing in the last strand and see:
\begin{multline}\nonumber
v_3(K(k_1,\dots,k_{2g+1})) - v_3(K(k_1,\dots,k_{2g+1}-1)) = -\frac{1}{2}\Big(a_2(K(k_1,\dots,k_{2g+1})) \\ + a_2(K(k_1,\dots,k_{2g+1}-1))+(s_{1,2g}+g)^2\Big)
\end{multline}
Repeatedly applying (\ref{eq:skein3}) to the rest of the crossings in the last strand, one finds that:
\begin{multline}\nonumber
v_3(K(k_1,\dots,k_{2g+1})) - v_3(K(k_1,\dots,k_{2g},0)) = \\-\frac{1}{2}\Big(a_2(K(k_1,\dots,k_{2g+1})) + a_2(K(k_1,\dots,k_{2g},0)) \\+ 2\sum_{j=1}^{k_{2g+1}-1}a_2(K(k_1,\dots,k_{2g},j))+k_{2g+1}(s_{1,2g}+g)^2\Big)
\end{multline}
Applying the formula we have obtained for $a_2$ together with Lemma \ref{lem:sym} gives:
\begin{multline}\nonumber
v_3(K(k_1,\dots,k_{2g+1})) - v_3(K(k_1,\dots,k_{2g},0)) = \\-\frac{1}{2}\bigg(\frac{1}{2}g(g+1) + g(s_{1,2g}+k_{2g+1}) + s_{2,2g}+k_{2g+1}s_{1,2g} + \frac{1}{2}g(g+1) + gs_{1,2g} + s_{2,2g} \\+ 2\sum_{j=1}^{k_{2g+1}-1}\left(\frac{1}{2}g(g+1) + g(s_{1,2g}+j) + s_{2,2g}+js_{1,2g}\right)+k_{2g+1}(s_{1,2g}+g)^2\bigg) \\
=-\frac{1}{2}\Big(g(g+1) + 2gs_{1,2g}+ gk_{2g+1} + 2s_{2,2g}+k_{2g+1}s_{1,2g}\\+ (k_{2g+1}-1)g(g+1) + 2g(k_{2g+1}-1)s_{1,2g}+2(k_{2g+1}-1)s_{2,2g} \\+gk_{2g+1}(k_{2g+1}-1) + k_{2g+1}(k_{2g+1}-1)s_{1,2g}+k_{2g+1}(s_{1,2g}+g)^2\Big)\\
=-\frac{1}{2}\Big(k_{2g+1}g(g+1) + 2gk_{2g+1}s_{1,2g}+ gk_{2g+1}^2  + 2k_{2g+1}s_{2,2g} \\+k_{2g+1}^2s_{1,2g} + k_{2g+1}\left(s_{1,2g}^2+2gs_{1,2g}+g^2\right)\Big)\\
=-\frac{1}{2}k_{2g+1}\Big(g(2g+1) + 4gs_{1,2g}+ gk_{2g+1} + 2s_{2,2g} +k_{2g+1}s_{1,2g} + s_{1,2g}^2\Big)
\end{multline}
Similarly, applying \eqref{eq:skein3} repeatedly to the $2g$th strand now gives:
\begin{multline}\label{eq:ind3}
v_3(K(k_1,\dots,k_{2g+1})) - v_3(K(k_1,\dots,k_{2g-1},0,0)) = \\
-\frac{1}{2}k_{2g}\Big(g(2g+1) + 4gs_{1,2g-1}+ gk_{2g} + 2s_{2,2g-1} +k_{2g}s_{1,2g-1} + s_{1,2g-1}^2\Big)\\
-\frac{1}{2}k_{2g+1}\Big(g(2g+1) + 4gs_{1,2g}+ gk_{2g+1} + 2s_{2,2g} +k_{2g+1}s_{1,2g} + s_{1,2g}^2\Big)
\end{multline}
Now, applying \eqref{eq:skein3} to one of the two rightmost crossings of $K(k_1,\dots,k_{2g-1},0,0)$, one sees that:
\begin{multline}\nonumber
v_3(K(k_1,\dots,k_{2g-1},0,0)) - v_3(K(k_1,\dots,k_{2g-1})) = \\
-\frac{1}{2}\left(a_2(K(k_1,\dots,k_{2g-1},0,0))+a_2(K(k_1,\dots,k_{2g-1})+(s_{1,2g-1}+g)^2\right) \\
=-\frac{1}{2}\bigg(\frac{g(g+1)}{2} + gs_{1,2g-1} + s_{2,2g-1}+\frac{g(g-1)}{2} + (g-1)s_{1,2g-1} \\+ s_{2,2g-1}+s_{1,2g-1}^2 + 2gs_{1,2g-1}+g^2\bigg)\\
=-\frac{1}{2}\left(2g^2 + (4g-1)s_{1,2g-1} + 2s_{2,2g-1} + s_{1,2g}^2\right)\\
\end{multline}
Combining this with \eqref{eq:ind3} and applying the induction hypothesis gives:
\begin{multline}\nonumber
v_3(K(k_1,\dots,k_{2g+1})) =  v_3(K(k_1,\dots,k_{2g-1})) -\frac{1}{2}\left(2g^2 + (4g-1)s_{1,2g-1} + 2s_{2,2g-1} + s_{1,2g-1}^2\right) \\
-\frac{1}{2}k_{2g}\left(g(2g+1) + 4gs_{1,2g-1}+ gk_{2g} + 2s_{2,2g-1} +k_{2g}s_{1,2g-1} + s_{1,2g-1}^2\right)\\
-\frac{1}{2}k_{2g+1}\left(g(2g+1) + 4gs_{1,2g}+ gk_{2g+1} + 2s_{2,2g} +k_{2g+1}s_{1,2g} + s_{1,2g}^2\right)\\
=-\frac{1}{2}\bigg(\frac{g(g-1)(2g-1)}{3} + (g-1)(2g-1)s_{1,2g-1} + (g-1) s_{1,2g-1}^2 + 2(g-1) s_{2,2g-1} + s_{1,2g-1}s_{2,2g-1} + s_{3,2g-1}\\+2g^2 + (4g-1)s_{1,2g-1} + 2s_{2,2g-1} + s_{1,2g-1}^2\\+ k_{2g}\left(g(2g+1) + 4gs_{1,2g-1}+ gk_{2g} + 2s_{2,2g-1} +k_{2g}s_{1,2g-1} + s_{1,2g-1}^2\right) \\+k_{2g+1}\left(g(2g+1) + 4gs_{1,2g}+ gk_{2g+1} + 2s_{2,2g} +k_{2g+1}s_{1,2g} + s_{1,2g}^2\right)\bigg)\\
=-\frac{1}{2}\bigg(\frac{g(g+1)(2g+1)}{3} + g(2g+1)s_{1,2g-1} + g s_{1,2g-1}^2 + 2g s_{2,2g-1} + s_{1,2g-1}s_{2,2g-1} + s_{3,2g-1}\\
+ g(2g+1)k_{2g} + 4gk_{2g}s_{1,2g-1}+ gk_{2g}^2+ 2k_{2g}s_{2,2g-1} +k_{2g}^2s_{1,2g-1} + k_{2g}s_{1,2g-1}^2 \\
+g(2g+1)k_{2g+1} + 4gk_{2g+1}s_{1,2g}+ gk_{2g+1}^2 + 2k_{2g+1}s_{2,2g} +k_{2g+1}^2s_{1,2g} + k_{2g+1}s_{1,2g}^2\bigg)\\
=-\frac{1}{2}\bigg(\frac{g(g+1)(2g+1)}{3} + g(2g+1)(s_{1,2g-1}+ k_{2g} + k_{2g+1}) \\+ g \left(s_{1,2g-1}^2 + 2k_{2g}s_{1,2g-1} +k_{2g}^2 + 2k_{2g+1}s_{1,2g} +  k_{2g+1}^2\right)\\ + 2g( s_{2,2g-1} + k_{2g}s_{1,2g-1} + k_{2g+1}s_{1,2g}) + s_{3,2g-1}
 + k_{2g}s_{2,2g-1} + k_{2g+1}s_{2,2g} \\ 
+ s_{1,2g-1}s_{2,2g-1}+k_{2g}s_{1,2g-1}(s_{1,2g-1} + k_{2g}) + k_{2g}s_{2,2g-1}  + k_{2g+1}s_{2,2g} +k_{2g+1}s_{1,2g}(s_{1,2g} + k_{2g+1})\bigg)
\end{multline}
We now repeatedly apply Lemma \ref{lem:sym} to find:
\begin{multline}\nonumber
v_3(K(k_1,\dots,k_{2g+1})) = -\frac{1}{2}\bigg(\frac{g(g+1)(2g+1)}{3} + g(2g+1)s_{1,2g+1} \\+ g \left((s_{1,2g-1}+k_{2g})^2 + 2k_{2g+1}s_{1,2g} +  k_{2g+1}^2\right) + 2g( s_{2,2g}+ k_{2g+1}s_{1,2g}) + s_{3,2g}+ k_{2g+1}s_{2,2g} \\
+ s_{1,2g-1}s_{2,2g-1}+ k_{2g}s_{2,2g-1}+k_{2g}s_{1,2g-1}s_{1,2g}   + k_{2g+1}s_{2,2g} +k_{2g+1}s_{1,2g}s_{1,2g+1}\bigg)\\
= -\frac{1}{2}\bigg(\frac{g(g+1)(2g+1)}{3} + g(2g+1)s_{1,2g+1} + g \left(s_{1,2g}^2 + 2k_{2g+1}s_{1,2g} +  k_{2g+1}^2\right) + 2gs_{2,2g+1} + s_{3,2g+1} \\ 
+ s_{2,2g-1}(s_{1,2g-1}+ k_{2g})+k_{2g}s_{1,2g-1}s_{1,2g}   + k_{2g+1}s_{2,2g} +k_{2g+1}s_{1,2g}s_{1,2g+1}\bigg)\\
= -\frac{1}{2}\bigg(\frac{g(g+1)(2g+1)}{3} + g(2g+1)s_{1,2g+1} + g (s_{1,2g} + k_{2g+1})^2 + 2gs_{2,2g+1} + s_{3,2g+1} \\ 
+ s_{2,2g-1}s_{1,2g}+k_{2g}s_{1,2g-1}s_{1,2g}   + k_{2g+1}s_{2,2g} +k_{2g+1}s_{1,2g}s_{1,2g+1}\bigg)\\
= -\frac{1}{2}\bigg(\frac{g(g+1)(2g+1)}{3} + g(2g+1)s_{1,2g+1} + gs_{1,2g+1}^2 + 2gs_{2,2g+1} + s_{3,2g+1} \\ 
+ s_{1,2g}(s_{2,2g-1}+k_{2g}s_{1,2g-1})   + k_{2g+1}s_{2,2g} +k_{2g+1}s_{1,2g}s_{1,2g+1}\bigg)\\
= -\frac{1}{2}\bigg(\frac{g(g+1)(2g+1)}{3} + g(2g+1)s_{1,2g+1} + gs_{1,2g+1}^2 + 2gs_{2,2g+1} + s_{3,2g+1} \\ 
+ s_{1,2g}s_{2,2g} + k_{2g+1}s_{2,2g} +k_{2g+1}s_{1,2g}s_{1,2g+1}\bigg)\\
= -\frac{1}{2}\bigg(\frac{g(g+1)(2g+1)}{3} + g(2g+1)s_{1,2g+1} + gs_{1,2g+1}^2 + 2gs_{2,2g+1} + s_{3,2g+1} \\ 
+ s_{2,2g}(s_{1,2g} + k_{2g+1}) +k_{2g+1}s_{1,2g}s_{1,2g+1}\bigg)\\
= -\frac{1}{2}\bigg(\frac{g(g+1)(2g+1)}{3} + g(2g+1)s_{1,2g+1} + gs_{1,2g+1}^2 + 2gs_{2,2g+1} + s_{3,2g+1} \\ 
+ s_{2,2g}s_{1,2g+1} +k_{2g+1}s_{1,2g}s_{1,2g+1}\bigg)\\
= -\frac{1}{2}\bigg(\frac{g(g+1)(2g+1)}{3} + g(2g+1)s_{1,2g+1} + gs_{1,2g+1}^2 + 2gs_{2,2g+1} + s_{3,2g+1} \\ 
+ s_{1,2g+1}(s_{2,2g} +k_{2g+1}s_{1,2g})\bigg)\\
= -\frac{1}{2}\left(\frac{g(g+1)(2g+1)}{3} + g(2g+1)s_{1,2g+1} + gs_{1,2g+1}^2 + 2gs_{2,2g+1} + s_{3,2g+1}
+ s_{1,2g+1}s_{2,2g+1}\right)
\end{multline}
as desired.
\end{proof}

\subsection{Cosmetic Surgery Constraints}
We now review some of the main constraints on the existence of chirally cosmetic surgeries involving finite type invariants that are known.
Studying the degree 2 part of the LMO invariant, Ito has obtained the following:
\begin{thm}[Corollary 1.3 of \cite{Ito}]\label{thm:ft}
Let $K$ be a knot and suppose $S^3_{p/q}(K) \cong \pm S^3_{p/q'}(K).$  If $v_3(K) \neq 0,$ then
\[
\frac{p}{q+q'} = \frac{7a_2(K)^2 - a_2(K)-10a_4(K)}{2v_3(K)}
\]
\end{thm}
\begin{proof}[Remark.]\let\qed\relax
Ito uses a slightly different definition for $v_3$ than is used here; in particular, he normalizes it to take the value 1/4 on the right-handed trefoil. This means that ours differs from his by a factor of 4, and in fact, the statement in \cite{Ito} has $8v_3$ in the denominator of the right-hand-side instead.
\end{proof}

On the other hand, by combining Casson-Gordon and Casson-Walker invariants, Ichihara, Ito, and Saito find:
\begin{thm}[Theorem 1.2 of \cite{IIS}]\label{thm:cass}
Let $K$ be a knot and suppose $S^3_{p/q}(K) \cong - S^3_{p/q'}(K).$ Then
\[
4(q+q')a_2(K)=-\sigma(K,p)
\]
\end{thm}

Here $\sigma(K,p)$ is the \textit{p-signature} of the knot $K$.  We briefly recall the definition and basic properties of the signature.  If $A$ is any matrix that represents the Seifert form of $K$, then for each $\omega \in S^1 \subseteq \mathbf{C}$ on the complex unit circle, we put $\sigma_{\omega}(K)$ equal to the signature of the Hermitian matrix $(1-\omega)A+(1-\overline{\omega})A^T$.  This is the so-called \textit{signature function} of $K$ on the unit circle.  It is integer-valued, and it only has discontinuities at roots of the Alexander polynomial, which can be seen by rewriting the matrix in the definition as $(\overline{\omega}-1)\left(\omega A-A^T\right)$.  Near $\omega = 1$, this function is zero, and near $\omega = -1$, this function is equal to and invariant called simply the \textit{signature} of the knot.  In our case, for $K=K(k_1,\dots,k_{2g+1})$ it is known that the signature is equal to $2g$.  As the ``jump" in the signature function is equal to two at a simple root of the Alexander polynomial, whenever such a knot has Alexander polynomial with no repeated roots (as will be the case for the relevant knots below), the intervals between the roots will attain all nonnegative even integers up to $2g$ as signatures; see Figure \ref{fig:sig} for an illustration of the case $K=K(1,0,0,0,0)$.  Finally, the \textit{p-signature} of a knot, denoted $\sigma(K,p)$ is equal to the sum of the values signature function at the $p$th roots of unity.  That is,
\[
\sigma(K,p) = \sum_{\omega^p=1} \sigma_{\omega}(K)
\]

\begin{figure}
\centering
\includegraphics[width=.5\textwidth]{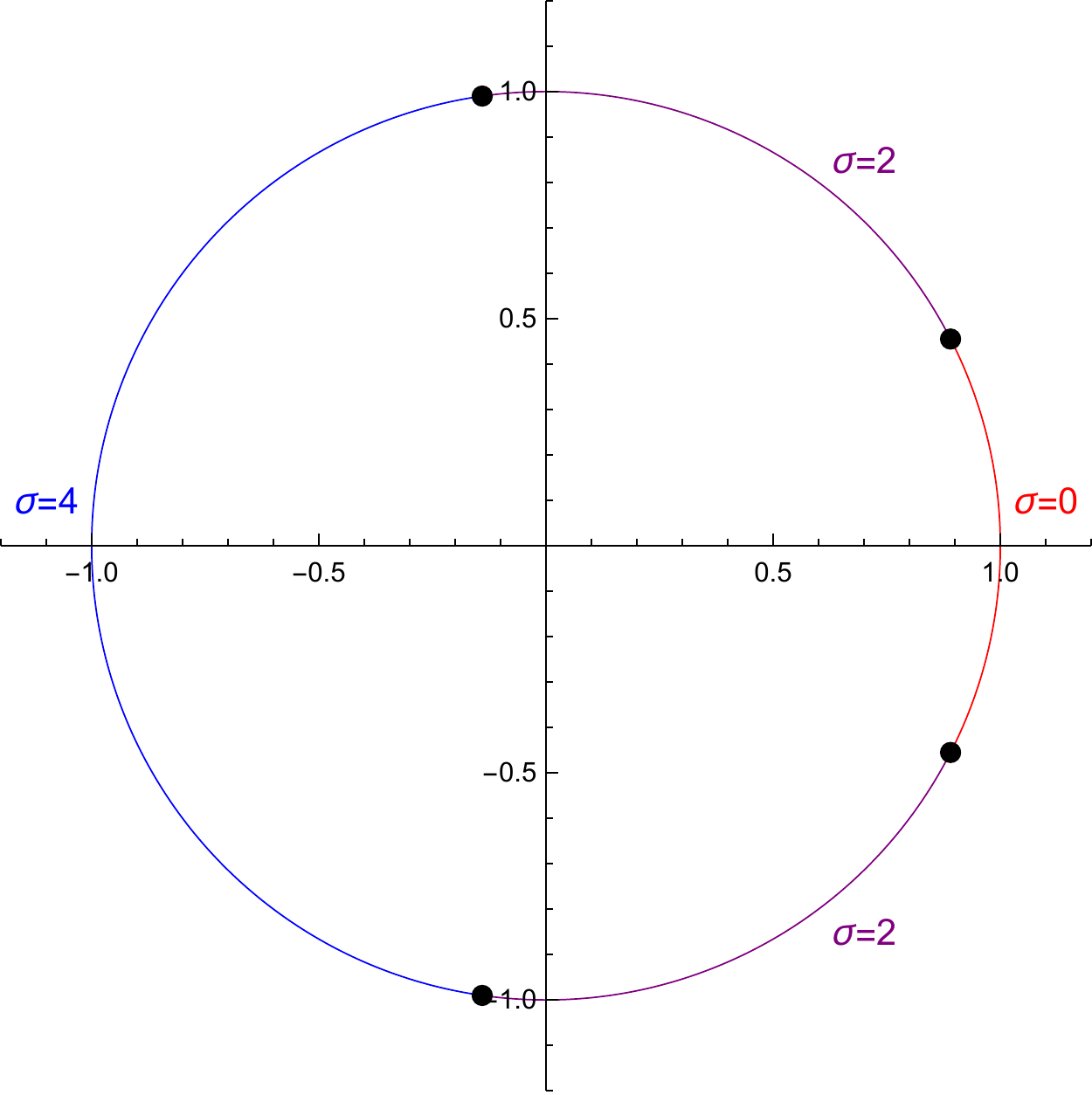}
\caption{The signature function for the knot $K(1,0,0,0,0)$.  The marked points are roots of the Alexander polynomial.}
\label{fig:sig}
\end{figure}

Following \cite{IIS}, we combine Theorems \ref{thm:ft} and \ref{thm:cass} to find that, if a knot $K$ admits a pair of chirally cosmetic surgeries along slopes $p/q$ and $p/q'$, then
\[
-\frac{\sigma(K,p)}{4a_2(K)} = q+q'=\frac{2pv_3(K)}{7a_2(K)^2 - a_2(K)-10a_4(K)}
\]
So we have obtained the following useful criterion:
\begin{cor}\label{cor:strong}
Let $K$ be a knot and suppose $S^3_{p/q}(K) \cong - S^3_{p/q'}(K).$ Then
\[
\frac{\sigma(K,p)}{p} = \frac{-8a_2(K)v_3(K)}{7a_2(K)^2 - a_2(K)-10a_4(K)}
\]
\end{cor}

\section{Proof of Theorem \ref{thm:main}}
We now prove our main result:

\main*
\noindent
We treat the five-stranded (genus 2) case first:
\subsection{The genus 2 case: $K=K(k_1,k_2,k_3,k_4,k_5)$}
In this case, by \eqref{eq:a242} and Lemma \ref{lem:av}

\begin{equation}\label{eq:ft2}
\begin{cases}
a_2(K) = 3 + 2 s_{1,5} + s_{2,5}\\
a_4(K) = 1 + s_{1,5} + s_{2,5} + s_{3,5} + s_{4,5}\\
v_3(K) = -\frac{1}{2}\left(10+10s_{1,5}+2s_{1,5}^2+4s_{2,5}+s_{1,5}s_{2,5}+s_{3,5}\right)
\end{cases}
\end{equation}

We see that $a_2(K),a_4(K)>0$ and $v_3(K)<0.$  Moreover, we also have that $0<\sigma(K,p)\leq 4p$ for all $p>0$.  So if $K$ admits chirally cosmetic surgeries with slopes $p/q$ and $p/q',$ then, by Corollary \ref{cor:strong},
\begin{align*}
0 < \frac{\sigma(K,p)}{p} &= \frac{-8a_2(K)v_3(K)}{7a_2(K)^2 - a_2(K)-10a_4(K)} \leq 4 \\
\leadsto 2a_2(K)|v_3(K)| & \leq 7a_2(K)^2 - a_2(K)-10a_4(K)\numberthis \label{eq:ineq}
\end{align*}
We also have, by \eqref{eq:ft2} the following estimate relating $a_2(K)$ and $a_4(K)$:
\[
a_4(K) = 1 + s_{1,5} + s_{2,5} + s_{3,5} + s_{4,5} \geq 1+ \frac{2}{3}s_{1,5} + \frac{1}{3}s_{2,5} = \frac{1}{3}a_2(K)
\]
Hence, if $K$ admits any chirally cosmetic surgery, then:
\begin{align*}
2a_2(K)|v_3(K)| & \leq 7a_2(K)^2 - a_2(K)-10a_4(K) < 7a_2(K)^2 - a_2(K)-3a_2(K) \\
\leadsto 2|v_3(K)| &< 7a_2(K) - 4
\end{align*}
Thus, we have shown:
\begin{cor}\label{cor:weak}
Let $K=K(k_1,k_2,k_3,k_4,k_5)$.  If
\[
2|v_3(K)| \geq 7a_2(K) - 4
\]
then $K$ does not admit any chirally cosmetic surgeries.
\end{cor}
We now apply this corollary to exclude all but four knots in this family from admitting chirally cosmetic surgeries.  We proceed in three cases:\\
\textbf{Case I: At least three $k_i$ are nonzero:}
By \eqref{eq:ft2} we see that
\begin{align*}
2|v_3(K)| - 7a_2(K) + 4 &= 10 +	10 s_1 + 2 s_1^2 +4 s_2 +s_1 s_2 +s_3 - 21 -14 s_1 -7 s_2 +4\\
&=-7-4s_1+2s_1^2-3s_2+s_1s_2+s_3\\
&=-7+(2s_1-4)s_1+(s_1-3)s_2 + s_3
\end{align*}

In this case, we have that $s_1 \geq 3$ and $s_3 \geq 1$ so that
\[
2|v_3(K)| - 7a_2(K) + 4 \geq -7+(6-4)3 + 1 =0
\]
Hence, by Corollary \ref{cor:weak}, these knots admit no chirally cosmetic surgeries.\\
\textbf{Case II: Exactly two $k_i$ are nonzero:}
Notice that, without loss of generality, we may assume $(k_1,k_2,k_3,k_4,k_5) = (a,b,0,0,0)$ for $a \geq b >0$.  This is because, by the symmetries of pretzel knots, one may cyclically permute the strands until the longest strand is in the first position; then, by applying ``flyping" moves, one can bring the other ``long" strand to the second position.  We compute that $s_1 = a+b,$ $s_2=ab,$ and $s_3=0.$

Suppose $a+b\geq 4.$ Then
\begin{align*}
|v_3(K)| &= 5+5(a+b)+(a+b)^2 + 2ab+ \frac{1}{2}(a+b)(ab) \\
& \geq 5 + 9(a+b) + \frac{ab}{2}(4+a+b)
\end{align*}
Thus
\begin{align*}
2|v_3(K)| - 7a_2(K) &\geq 10 + 18(a+b) + ab(4+a+b) - 21 -14 (a+b) -7 ab +4\\
&=-7+4(a+b)+ab(a+b-3)\\
&geq -7+16 + ab \geq 0
\end{align*}
Once again, Corollary \ref{cor:weak} guarantees that $K(a,b,0,0,0)$ admits no chirally cosmetic surgeries as long as $a+b \geq 4.$  This leaves: $K(1,1,0,0,0)$ and $K(2,1,0,0,0)$ as still unchecked; we deal with these two knots below.\\
\textbf{Case III: Exactly one $k_i$ nonzero:}
As before, we may assume (by appropriate permutation) that the longest strand is in the first position; i.e., we are considering the case $K(N,0,0,0,0)$ with $N > 0.$  Here $s_1=N$ and $s_2=s_3=s_4=0.$  Thus, by \eqref{eq:ft2}:
\begin{align*}
7a_2(K)^2 - a_2(K) - 10a_4(K) &= 28N^2+72N+50 \\
2a_2(K)|v_3(K)| &= 4N^3+26N^2+50N+30
\end{align*}
We see that \eqref{eq:ineq} is satisfied only for $N \leq 2,$ and so, by Corollary \ref{cor:strong}, $K(N,0,0,0,0)$ admits no chirally cosmetic sugeries when $N \geq 3.$  This leaves $K(1,0,0,0,0)$ and $K(2,0,0,0,0).$  As still unchecked.  We turn to these and the other two leftover knots now.\\
\textbf{The remaining knots:}
Let us turn our attention to the four knots that still need to be checked.

First, we consider $K=K(2,1,0,0,0).$  We compute, using \eqref{eq:ft2} that
\begin{align*}
2a_2(K)|v_3(K)| &= 2 (11)(36) = 792\\
7a_2(K)^2 - a_2(K) - 10a_4(K) &= 7(121) -11-60 = 776
\end{align*}
The inequality \eqref{eq:ineq} is not satisfied, and so, by Corollary \ref{cor:strong}, this knot does not admit any chirally cosmetic surgeries.

For the next three knots, we will investigate their signature functions in order to obtain bounds for their respective p-signatures.

Let $K=K(1,0,0,0,0).$  We compute that $a_2(K)=5$, $a_4(K)=2$, and $v_3(K)=-11$
\[
\frac{-8a_2(K)v_3(K)}{7a_2(K)^2 - a_2(K)-10a_4(K)} = \frac{44}{15}.
\]
Hence, by Corollary \ref{cor:strong}, if $K$ admits chirally cosmetic surgeries with slopes $p/q$ and $p/q',$ then $p=15n$ for some positive integer $n$.  From the Conway polynomial $1+5z^2+2z^4,$ we find that the roots of the Alexander polynomial are $e^{\pm i\theta_1}$ and $e^{\pm i\theta_2},$ where $\theta_1 \approx 1.712>\frac{8\pi}{15}$ and $\theta_2 \approx 0.473>\frac{2\pi}{15}.$  Hence, $\sigma_\omega(K)=0$ when $-\frac{2\pi}{15}\leq \arg(\omega)\leq \frac{2\pi}{15}$ and $\sigma_\omega(K)\leq 2$ when $-\frac{8\pi}{15}\leq \arg(\omega)\leq \frac{8\pi}{15}.$  See Figure \ref{fig:sig} for an illustration.

So
\begin{align*}
\frac{\sigma(K,15n)}{15n} &\leq \frac{2(6n)+4(7n-1)}{15n}\\
&=\frac{40n-4}{15n}\\
&<\frac{40}{15} < \frac{44}{15}
\end{align*}
Hence, by Corollary \ref{cor:strong}, $K$ admits no chirally cosmetic surgeries.

Now let $K=K(2,0,0,0,0).$  We compute $a_2(K)=7$, $a_4(K)=3$, and $v_3(K)=-19$
\[
\frac{-8a_2(K)v_3(K)}{7a_2(K)^2 - a_2(K)-10a_4(K)} = \frac{532}{153}.
\]
As before, by Corollary \ref{cor:strong}, if $K$ admits chirally cosmetic surgeries with slopes $p/q$ and $p/q',$ then $p=153n$ for some positive integer $n$.  By \eqref{eq:a242}, this knot has Conway polynomial $1+7z^2+3z^4$.  Thus we find that the roots of the Alexander polynomial, which we denote $e^{\pm i\theta_1}$ and $e^{\pm i\theta_2}$, satisfy $\theta_1 \approx 1.661>\frac{80\pi}{153}$ and $\theta_2 \approx 0.394>\frac{18\pi}{153}$.  Hence, $\sigma_\omega(K)=0$ when $-\frac{18\pi}{153}\leq \arg(\omega)\leq \frac{18\pi}{153}$ and $\sigma_\omega(K)\leq 2$ when $-\frac{80\pi}{153}\leq \arg(\omega)\leq \frac{80\pi}{153}$.
So
\begin{align*}
\frac{\sigma(K,153n)}{153n} &\leq \frac{2(62n)+4(73n-1)}{15n}\\
&=\frac{416n-4}{153n}\\
&<\frac{416}{153} < \frac{532}{153}
\end{align*}
Hence, by Corollary \ref{cor:strong}, $K$ admits no chirally cosmetic surgeries.

Lastly, let $K=K(1,1,0,0,0).$  We compute $a_2(K)=8$, $a_4(K)=4$, and $v_3(K)=-22$
\[
\frac{-8a_2(K)v_3(K)}{7a_2(K)^2 - a_2(K)-10a_4(K)} = \frac{88}{25}.
\]
As before, by Corollary \ref{cor:strong}, if $K$ admits chirally cosmetic surgeries with slopes $p/q$ and $p/q',$ then $p=25n$ for some positive integer $n$.  From the Conway polynomial $1+8z^2+4z^4,$ we find that the roots of the Alexander polynomial are $e^{\pm i\theta_1}$ and $e^{\pm i\theta_2},$ where $\theta_1 \approx 1.504>\frac{10\pi}{25}$ and $\theta_2 \approx 0.368>\frac{2\pi}{25}.$  Hence, $\sigma_\omega(K)=0$ when $-\frac{2\pi}{25}\leq \arg(\omega)\leq \frac{2\pi}{25}$ and $\sigma_\omega(K)\leq 2$ when $-\frac{10\pi}{25}\leq \arg(\omega)\leq \frac{10\pi}{25}.$
So
\begin{align*}
\frac{\sigma(K,25n)}{25n} &\leq \frac{2(8n)+4(15n-1)}{15n}\\
&=\frac{76n-4}{25n}\\
&<\frac{76}{25} < \frac{88}{25}
\end{align*}
Once again, by Corollary \ref{cor:strong}, $K$ admits no chirally cosmetic surgeries.  This concludes the case of genus 2.

\subsection{The genus 3 case: $K=K(k_1,k_2,k_3,k_4,k_5,k_6,k_7)$}
In this case, by \eqref{eq:a243}, and Lemma \ref{lem:av}

\begin{equation}\label{eq:ft3}
\begin{cases}
a_2(K) = 6 + 3 s_{1,7} + s_{2,7}\\
a_4(K) = 5 + 4s_{1,7} + 3s_{2,7} + 2s_{3,7} + s_{4,7}\\
a_6(K) = 1 + s_{1,7} + s_{2,7} + s_{3,7} + s_{4,7} + s_{5,7} + s_{6,7}\\
v_3(K) = -\frac{1}{2}\left(28+21s_{1,5}+3s_{1,5}^2+6s_{2,5}+s_{1,5}s_{2,5}+s_{3,5}\right)
\end{cases}
\end{equation}

Once again $a_2(K),a_4(K)>0$ and $v_3(K)<0.$  Moreover, we also have that $0<\sigma(K,p)\leq 6p$ for all $p>0$.  So if $K$ admits chirally cosmetic surgeries with slopes $p/q$ and $p/q',$ then, by Corollary \ref{cor:strong},
\begin{align*}
0 < \frac{\sigma(K,p)}{p} &= \frac{-8a_2(K)v_3(K)}{7a_2(K)^2 - a_2(K)-10a_4(K)} \leq 6 \\
\leadsto \frac{4}{3}a_2(K)|v_3(K)| & \leq 7a_2(K)^2 - a_2(K)-10a_4(K)\numberthis \label{eq:ineq3}
\end{align*}
We also have, by \eqref{eq:ft3} the following estimate relating $a_2(K)$ and $a_4(K)$:
\[
\frac{6}{5}a_4(K) = 6 + \frac{24}{5}s_{1,7} + \frac{18}{5}s_{2,7} + \frac{12}{5}s_{3,7} + \frac{6}{5}s_{4,7} \geq 6+ 3s_{1,7} + s_{2,7} = a_2(K)
\]
Hence, if $K$ admits any chirally cosmetic surgery, then:
\begin{align*}
\frac{4}{3}a_2(K)|v_3(K)| & \leq 7a_2(K)^2 - a_2(K)-10a_4(K) \leq 7a_2(K)^2 - a_2(K)-\frac{25}{3}a_2(K) \\
\leadsto 4|v_3(K)| & \leq  21a_2(K) - 28
\end{align*}
Thus, we have shown:
\begin{cor}\label{cor:weak3}
Let $K=K(k_1,k_2,k_3,k_4,k_5,k_6,k_7)$.  If
\[
4|v_3(K)| > 21a_2(K) - 28
\]
then $K$ does not admit any chirally cosmetic surgeries.
\end{cor}
Now we compute:
\begin{align*}
4|v_3(K)| - 21a_2(K)+ 28 &=-42-21s_{1,7}-9s_{2,7}+6s_{1,7}^2+2s_{1,7}s_{2,7}+2s_{3,7}\\
&=-42 + 3s_{1,7}(2s_{1,7}-7) + s_{2,7}(2s_{1,7}-9)+2s_{3,7}
\end{align*}
If $s_{1,7} \geq 5$, then
\[
4|v_3(K)| - 21a_2(K)+ 28 \geq -42+3(5)(3) > 0
\]
Hence, by Corollary \ref{cor:weak3}, these knots do not admit any chirally cosmetic surgeries.  It remains to check the knots in this family with $s_{1,7}<5$.  These are (once again using the symmetries of the pretzel knots): $K(1,1,1,1,0,0,0)$, $K(2,1,1,0,0,0,0)$, $K(1,1,1,0,0,0,0)$, $K(3,1,0,0,0,0,0)$, $K(2,2,0,0,0,0,0)$, $K(2,1,0,0,0,0,0)$, $K(1,1,0,0,0,0,0)$, $K(4,0,0,0,0,0,0)$, $K(3,0,0,0,0,0,0)$, $K(2,0,0,0,0,0,0)$, and $K(1,0,0,0,0,0,0)$.

First, we note that for $K=K(1,1,1,1,0,0,0)$, we have, by \eqref{eq:ft3}, $a_2(K)=24$, $a_4(K)=48$, and $v_3(K)=-112$.  Thus, $\frac{4}{3}a_2(K)|v_3(K)| = 3584 > 3528 = 7a_2(K)^2 - a_2(K)-10a_4(K)$ so that inequality \ref{eq:ineq3} is not satisfied, implying that this knot admits no chirally cosmetic surgeries.

For the rest of the knots, as in the genus 2 case, we compute $F \coloneqq \frac{-8a_2(K)v_3(K)}{7a_2(K)^2 - a_2(K)-10a_4(K)}$.  By Corollary \ref{cor:strong}, if $K$ admits chirally cosmetic surgeries with slopes $p/q$ and $p/q'$, then $\frac{\sigma(K,p)}{p}=F$ so that $p$ must be a multiple of the denominator of $F$.   From the roots of their Alexander polynomials, we obtain bounds on their signature functions. For all but one of the remaining knots, we find that $\frac{\sigma(K,p)}{p}<F$ so that those knots admit no chirally cosmetic surgeries.  The results are summarized in the following table (we denote the roots of the Alexander polynomial by $e^{\pm i\theta_1}$,  $e^{\pm i\theta_2}$, and  $e^{\pm i\theta_3}$):\\
\begin{center}
\begin{tabular}{|c|c|c|c|c|c|}
\hline
$K$ & $F$ & $\theta_1 >$ & $\theta_2 >$ & $\theta_3 >$ & $\frac{\sigma(K,p)}{p}\leq$\\
\hline
$K(2,1,1,0,0,0,0)$ & $\frac{1219}{205}$ & $\frac{14\pi}{205}$ & $\frac{54\pi}{205}$ & $\frac{122\pi}{205}$ & $ \frac{6(83n-1)+4(68n)+2(40n)}{205n}<\frac{850}{205}$\\
\hline
$K(1,1,1,0,0,0,0)$ & $\frac{5256}{985}$ & $\frac{76\pi}{985}$ & $\frac{282\pi}{985}$ & $\frac{600\pi}{985}$ & $ \frac{6(385n-1)+4(318n)+2(206n)}{985n}<\frac{3994}{985}$\\
\hline
$K(3,1,0,0,0,0,0)$ & $\frac{2660}{461}$ & $\frac{32\pi}{461}$ & $\frac{140\pi}{461}$ & $\frac{294\pi}{461}$ & $ \frac{6(167n-1)+4(154n)+2(108n)}{461n}<\frac{1834}{461}$\\
\hline
$K(2,2,0,0,0,0,0)$ & $\frac{400}{69}$ & $\frac{4\pi}{69}$ & $\frac{20\pi}{69}$ & $\frac{42\pi}{69}$ & $ \frac{6(27n-1)+4(22n)+2(16n)}{69n}<\frac{282}{69}$\\
\hline
$K(2,1,0,0,0,0,0)$ & $\frac{578}{111}$ & $\frac{8\pi}{111}$ & $\frac{34\pi}{111}$ & $\frac{70\pi}{111}$ & $ \frac{6(41n-1)+4(36n)+2(26n)}{111n}<\frac{442}{111}$\\
\hline
$K(1,1,0,0,0,0,0)$ & $\frac{468}{101}$ & $\frac{8\pi}{101}$ & $\frac{32\pi}{101}$ & $\frac{64\pi}{101}$ & $ \frac{6(37n-1)+4(32n)+2(24n)}{101n}<\frac{398}{101}$\\
\hline
$K(4,0,0,0,0,0,0)$ & $\frac{96}{17}$ & $0$ & $\frac{4\pi}{17}$ & $\frac{10\pi}{17}$ & $ \frac{6(7n-1)+4(6n)+2(4n)}{17n}<\frac{74}{17}$\\
\hline
$K(3,0,0,0,0,0,0)$ & $\frac{708}{139}$ & $\frac{10\pi}{139}$ & $\frac{48\pi}{139}$ & $\frac{92\pi}{139}$ & $ \frac{6(47n-1)+4(44n)+2(38n)}{139n}<\frac{534}{139}$\\
\hline
$K(2,0,0,0,0,0,0)$ & $\frac{1968}{433}$ & $\frac{42\pi}{433}$ & $\frac{158\pi}{433}$ & $\frac{292\pi}{433}$ & $ \frac{6(141n-1)+4(134n)+2(116n)}{433n}<\frac{1614}{433}$\\
\hline
\end{tabular}
\end{center}
\hfill \\

Finally, we turn our attention to the knot $K(1,0,0,0,0,0,0)$.  By \eqref{eq:ft3}, we compute that: $\frac{-8a_2(K)v_3(K)}{7a_2(K)^2 - a_2(K)-10a_4(K)} = 4$.  The roots of the Alexander polynomial for this knot are $e^{\pm 2\pi i\alpha}$,  $e^{\pm 2\pi i\beta}$, and  $e^{\pm 2\pi i\gamma}$, where $.056 < \alpha < .057$, $.19 < \beta < .191$, and $.342 < \gamma < .343$.  Therefore, we obtain the estimate:
\begin{align*}
\frac{\sigma(K,p)}{p} &\leq \frac{6\lceil .316 p\rceil + 4(2\lceil.153p\rceil) + 2(2\lceil.135 p\rceil)}{p}\\
&\leq \frac{6(.316 p +1) + 4(2(.153p+1)) + 2(2(.135 p +1))}{p}\\
&= 3.66+\frac{18}{p}
\end{align*}
Hence, if $p\geq 53$, then $\frac{\sigma(K,p)}{p}<4$, which by Corollary \ref{cor:strong} excludes the possibility of chirally cosmetic surgeries on this knot with slopes $p/q$ and $p/q'$ whenever $p \geq 53$.  Using Mathematica \cite{WM}, we  explicitly compute $\frac{\sigma(K,p)}{p}$ for $1 \leq p \leq 52$. The results are in the table below:
\begin{center}
\begin{tabular}{|c|c||c|c||c|c||c|c|}
\hline
$p$ & $\frac{\sigma(K,p)}{p}$ & $p$ & $\frac{\sigma(K,p)}{p}$ & $p$ & $\frac{\sigma(K,p)}{p}$ & $p$ & $\frac{\sigma(K,p)}{p}$\\
\hline
 1 & 0 & 14 & $\frac{27}{7}$ & 27 & $\frac{32}{9}$ &
   40 & $\frac{73}{20}$ \\ \hline
 2 & 3 & 15 & $\frac{56}{15}$ & 28 & $\frac{51}{14}$ &
   41 & $\frac{148}{41}$ \\ \hline
 3 & $\frac{8}{3}$ & 16 & $\frac{29}{8}$ &
   29 & $\frac{108}{29}$ & 42 & $\frac{25}{7}$ \\ \hline
 4 & $\frac{7}{2}$ & 17 & $\frac{64}{17}$ &
   30 & $\frac{11}{3}$ & 43 & $\frac{156}{43}$ \\ \hline
 5 & 4 & 18 & $\frac{31}{9}$ & 31 & $\frac{116}{31}$ &
   44 & $\frac{79}{22}$ \\ \hline
 6 & 3 & 19 & $\frac{68}{19}$ & 32 & $\frac{59}{16}$ &
   45 & $\frac{164}{45}$ \\ \hline
 7 & $\frac{24}{7}$ & 20 & $\frac{37}{10}$ &
   33 & $\frac{40}{11}$ & 46 & $\frac{85}{23}$ \\ \hline
 8 & $\frac{15}{4}$ & 21 & $\frac{24}{7}$ &
   34 & $\frac{63}{17}$ & 47 & $\frac{172}{47}$ \\ \hline
 9 & $\frac{32}{9}$ & 22 & $\frac{39}{11}$ &
   35 & $\frac{132}{35}$ & 48 & $\frac{29}{8}$ \\ \hline
 10 & $\frac{19}{5}$ & 23 & $\frac{84}{23}$ &
   36 & $\frac{65}{18}$ & 49 & $\frac{180}{49}$ \\ \hline
 11 & $\frac{40}{11}$ & 24 & $\frac{43}{12}$ &
   37 & $\frac{132}{37}$ & 50 & $\frac{91}{25}$ \\ \hline
 12 & $\frac{7}{2}$ & 25 & $\frac{92}{25}$ &
   38 & $\frac{67}{19}$ & 51 & $\frac{188}{51}$ \\ \hline
 13 & $\frac{48}{13}$ & 26 & $\frac{49}{13}$ &
   39 & $\frac{140}{39}$ & 52 & $\frac{97}{26}$ \\ \hline
\end{tabular}
\end{center}

\hfill \\
We see that $\frac{\sigma(K,p)}{p}=4$ only when $p=5$.  Hence, by Corollary \ref{cor:strong}, if $K$ admits chirally cosmetic surgeries, they must have surgery slopes $5/q$ and $5/q'$ for some integers $q$ and $q'$.  However, by Theorem \ref{thm:cass}, we must have:
\[
q+q' = \frac{-\sigma(K,5)}{4a_2(K)} = \frac{-20}{4\cdot 9} = -\frac{5}{9}
\]
which is impossible, as $q$ and $q'$ are supposed to be integers.
\qed

\section{Proof of Theorem \ref{thm:gen}}

Now we turn our attention to the general case.  We show that for a fixed genus, at most finitely many alternating odd pretzel knots can possibly admit chirally cosmetic surgeries.  We shall make use of the following:
\begin{lem}[Corollary 6.2 of \cite{IIS}]\label{lem:weak}
Let $K$ be a nontrivial negative knot (i.e., a knot with all negative crossings).  If $4|v_3(K)| \geq 7g a_2(K)$ then $K$ admits no chirally cosmetic surgeries (here $g$ is the Seifert genus of $K$.
\end{lem}
\begin{proof}This follows from the fact that for negative knots, $a_2(K),a_4(K) \geq 0$ \cite{Crom}, $v_3(K)<0$ \cite{Stoi}, and $\sigma(K,p)>0$ \cite{PT}.  (For the family of knots we are considering, two of these inequalities can be seen directly from Lemma \ref{lem:av}).  Now Corollary \ref{cor:strong} implies that if $p/q$- and $p/q'$-surgeries are chirally cosmetic,
\begin{align*}
0 < \frac{\sigma(K,p)}{p} &= \frac{-8a_2(K)v_3(K)}{7a_2(K)^2 - a_2(K)-10a_4(K)} \leq 2g \\
\leadsto 4a_2(K)|v_3(K)| & \leq g(7a_2(K)^2 - a_2(K)-10a_4(K) < 7a_2(K)^2)\\
\leadsto 4|v_3(K)| &< 7ga _2(K)
\end{align*}
\end{proof}
We now prove:
\gen*
\begin{proof}
As these knots are negative, by Lemma \ref{lem:weak}, $K$ admits no chirally cosmetic surgeries if $4|v_3(K)|\geq 7g a_2(K)$ .  By Lemma \ref{lem:av}, we compute:
\begin{multline}\nonumber
4|v_3(K)|- 7g a_2(K) = \frac{2g(g+1)(2g+1)}{3}-\frac{7g(g+1)}{2} + 2g(2g+1)s_{1,2g+1} - 7g^2s_{1,2g+1} \\+ 4gs_{2,2g+1}-7gs_{2,2g+1} + 2gs_{1,2g+1}^2+2s_{1,2g+1}s_{2,2g+1} + 2s_{3,2g+1}\\
= \frac{-13g^3-9g^2+4g}{6} + g(2-3g)s_{1,2g+1}+ 2gs_{1,2g+1}^2 -3gs_{2,2g+1}+2s_{1,2g+1}s_{2,2g+1} + 2s_{3,2g+1}\\
=\frac{-13g^3-9g^2+4g}{6} + gs_{1,2g+1}(2s_{1,2g+1}+2-3g) + s_{2,2g+1}(2s_{1,2g+1}-3g) + 2s_{3,2g+1}
\end{multline}
If $s_{1,2g+1} \geq \alpha g$, then $s_{1,2g+1}(2s_{1,2g+1}-3g)\geq \frac{13g^2}{6}$, $2s_{1,2g+1}-3g \geq 0$, and $2gs_{1,2g+1} \geq 4g^2$. Hence:
\[
4|v_3(K)|- 7g a_2(K) \geq \frac{-13g^3-9g^2+4g}{6} + \frac{13g^3}{6}+4g^2 \geq 0
\]
as desired.
\end{proof}
\printbibliography

\end{document}